\documentclass[12pt]{amsart} 

\usepackage[margin=2.9cm]{geometry}             
\usepackage{color}
\usepackage{graphicx}
\usepackage{amssymb, enumerate,bm}
\usepackage[matrix,arrow,ps]{xy}
\usepackage{cite}
\usepackage{epstopdf, amsmath}
\usepackage{paralist}
\newcommand{\C}{\mathbb C}

\newcommand{\transp}{\,^t}

\newtheorem{theo}{Theorem}[section]
\newtheorem{lemma}[theo]{Lemma}
\newtheorem{cor}[theo]{Corollary}

\theoremstyle{remark}
\newtheorem{remark}[theo]{Remark}
\theoremstyle{example}

\theoremstyle{definition}

\numberwithin{equation}{section}

\begin{document}

\begin{abstract}
We construct group-invariant CR maps from the unit sphere in $\C^3$ and provide sharp bounds for the  gap termination in this setting. 
\end{abstract}

\author{Mona Al Batrouni and Florian Bertrand }
\title[Group-Invariant CR maps of the sphere $\mathbb{S}^5$]{Group-Invariant CR maps of the sphere $\mathbb{S}^5$}



\maketitle 
\vspace{-3mm }


\section*{Introduction}
The study of CR maps between spheres is an important question in CR geometry. Nonconstant CR maps between spheres can be extended as proper holomorphic maps between the corresponding balls, and we know from Forstneri\v{c} \cite{fo1} that such 
maps, when assumed to be smooth, are necessarily rational. In the present paper, we focus on symmetric CR maps, that is, CR maps that are invariant under certain subgroups of the unitary group. Forstneri\v{c} in \cite{fo2} showed that there are restrictions on the considered   subgroups, 
and D'Angelo-Lichtblau \cite{da-li} characterized all possible subgroups that give rise to nonconstant invariant CR maps. In recent years, this question and its connection with the gap phenomenon has attracted a lot of attention 
\cite{da4,da2,da1,da3,dkr,da-le,brgr2,hu,hjy,brgr,eb} with important foundational work \cite{we,fa}. Some of these works rely on a clever interplay between CR geometry and Real Algebraic Geometry by means of special real-valued polynomials. In this direction, we emphasize the recent article \cite{bcggkmp} in which the authors construct nonconstant invariant CR maps between spheres and  study gap termination phenomena in that setting.  In particular, they obtain a sharp bound for the gap termination in the case of $\C^2$ and suggest that similar bounds may occur in $\C^n$. In the present paper, we address this question in the case of $\C^3$ and provide in Theorem \ref{theomain}  sharp bounds of order comparable to those obtained in Theorem 2 in \cite{bcggkmp}. 

Near completion of this work, we understood, after a discussion with some authors of \cite{bcggkmp}, that some of them and their students are currently investigating similar and related questions, with a work \cite{blmsw} being finalized. From our understanding, the results do not overlap and may coexist along each other.
\section{Preliminaries}

\subsection{Admissible subgroups, canonical CR maps and polynomials}

We denote by $\mathbb{S}^{2n-1}=\{(z_1,\ldots,z_n) \in \C^n \ | \ |z_1|^2+\ldots + |z_n|^2=1\}$ the unit sphere in $\C^n$ and by $U(n)$ the group of unitary matrices of size $n$.  

Let  $\Gamma$ be a finite and fixed point free subgroup of  $U(n)$. A CR map between spheres \\ $\varphi: \mathbb{S}^{2n-1} \to \mathbb{S}^{2N-1}$ is {\it $\Gamma$-invariant} if $\varphi \circ \gamma = \varphi$ for all $\gamma \in \Gamma$. Groups that admit such invariant smooth and nonconstant CR maps are entirely characterized by D'Angelo–Lichtblau in \cite{da-li}. More precisely, denoting by $I_d$ the identity matrix of size $d$, they must be equivalent to one of the following: 
\begin{equation*}
\begin{cases}
\langle\omega I_j \oplus\omega^2 I_{n-j}\rangle  \mbox{ where $\omega$ is a primitive $p^{th}$ root of unity, and $p$ an odd integer,}\\ 
\langle \omega I_n \rangle  \mbox{ where $\omega$ is a primitive $p^{th}$ root of unity,}\\
\langle \omega I_j \oplus \omega^2 I_k \oplus \omega^4 I_{n-j-k} \rangle \mbox{ where $\omega$ is a $7^{th}$ root of unity.} 
  \end{cases}
\end{equation*}
Following \cite{bcggkmp}, we call these subgroups {\it admissible}. 

\vspace{0.5cm}

Let $\Gamma$ be an admissible subgroup of $U(n)$. Following e.g. D'Angelo--Lichtblau in \cite{da-li} we now recall the construction of an important  canonical CR map between spheres (see also \cite{da2}). 
For $z,w\in \C^n$,  we denote by $\langle z, w \rangle = \transp \overline{z} w$ the usual Hermitian product in $\C^n$. We define the following polynomial
$$\Phi_{\Gamma}(z, \bar{z}) = 1 - \prod_{\gamma \in \Gamma} \left( 1 - \langle \gamma z, z \rangle \right).$$
Note that $\Phi_{\Gamma}$ is real-valued  and $\Gamma$-invariant, that is,  $\Phi_{\Gamma}(\gamma z, \overline{\gamma z})=\Phi_{\Gamma}(z, \bar{z})$ for any $\gamma \in \Gamma$ and $z\in \C^n$ (see Lemma 3 in \cite{da-li}). The coefficients of $\Phi_{\Gamma}$ define a 
positive semi-definite  Hermitian form. We denote by $N(\Gamma)$ its rank and we write 
$$\Phi_{\Gamma}(z, \bar{z})=\|\varphi_{\Gamma}(z)\|^2,$$ 
where $\varphi_{\Gamma}:  \mathbb{S}^{2n-1} \to \mathbb{S}^{2N(\Gamma)-1}$ is a nonconstant $\Gamma$-invariant holomorphic polynomial 
with $N(\Gamma)$ linearly independent components. We refer to  $\varphi_{\Gamma}$ as the  {\it canonical $\Gamma$-invariant CR map associated to $\Gamma$}. 

\vspace{0.5cm}

Questions related to CR maps between spheres and the map  $\varphi_{\Gamma}$ in particular, may be addressed by means of real-valued polynomials with nonnegative 
coefficients \cite{da1, da-li,da-le,bcggkmp}. Assume that $\varphi \colon \mathbb{C}^n \to \mathbb{C}^N$ is a holomorphic map with monomial  components satisfying $\varphi(\mathbb{S}^{2n-1}) \subset \mathbb{S}^{2N-1}$. We write  
$\varphi(z) = (\ldots, c_\alpha z^\alpha, \ldots)$
and  note that whenever $\sum_{j=1}^n |z_j|^2 = 1$ we have $\sum_{\alpha} |c_\alpha|^2 |z^\alpha|^2 = 1.$
It follows that the real-valued polynomial defined by 
$$P(x_1,\ldots,x_n):=\sum_{\alpha} |c_\alpha|^2 x^\alpha$$ is equal to $1$ in the hyperplane  
$x_1 + \ldots + x_n = 1$. The {\it rank of $P$}, that is, the number of linearly independent monomials, is equal to $N$. 
Moreover, if the map $\varphi$ is $\Gamma$-invariant, then the polynomial $P$ is also $\Gamma$-invariant. We then denote by $\mathcal{P}(\Gamma)$ the set of $\Gamma$-invariant real-valued polynomials $p(x_1,\ldots,x_n)$ with nonnegative coefficients and which are equal $1$ whenever $x_1 + \ldots + x_n = 1$. In turn, any polynomial  $p \in \mathcal{P}(\Gamma)$ of rank $N$ defines a  $\Gamma$-invariant CR map from 
$\mathbb{S}^{2n-1}$ to $\mathbb{S}^{2N-1}$. This correspondence is an important tool in order to construct explicitly nonconstant smooth $\Gamma$-invariant CR maps between spheres.

In case the map in consideration is the canonical CR map $\varphi_{\Gamma}$, we  denote by $f_{\Gamma}$ the corresponding polynomial, called the {\it canonical $\Gamma$-invariant polynomial}, and we note that 
$$f_{\Gamma}(x_1, \ldots, x_n) = f_{\Gamma}(|z_1|^2, \ldots, |z_n|^2)=\Phi_{\Gamma}(z, \bar{z}).$$
The rank of  $f_{\Gamma}$ is precisely $N(\Gamma)$.

\subsection{Review of  previous results}
The family of polynomials presented in the previous section was first discovered by D'angelo in \cite{da4} and their interplay with
proper holomorphic maps between balls or CR maps between spheres has been extensively studied since then; see for instance \cite{da1,da2,da3,dkr,da-le,bcggkmp}. In the context of $\Gamma$-invariant CR maps, we emphasize the following two recent results of \cite{bcggkmp}. Let $\Gamma$ be an admissible subgroup of $U(n)$. For any $N\geq N(\Gamma)^2-2N(\Gamma)+2$, there is a nonconstant smooth 
 $\Gamma$-invariant CR map $\varphi: \mathbb{S}^{2n-1} \to \mathbb{S}^{2N-1}$ for which $N$ is the minimal embedding 
dimension (Theorem 1 in \cite {bcggkmp}). The proof which relies on the correspondence between $\Gamma$-invariant CR maps 
and $\mathcal{P}(\Gamma)$ is adapted from the one given by D'Angelo-Lebl \cite{da-le} in the noninvariant case. The authors then sharpen the above bound in the case of the 
sphere $\mathbb{S}^3\subset \C^{2}$, showing that the minimal rank is then $2N(\Gamma)-1$ (Theorem 2 in \cite {bcggkmp}). The proof is based  on a 
well chosen sequence of iterations applied to the canonical $\Gamma$-invariant polynomial $f_{\Gamma}$ that allows to construct 
new polynomials $\mathcal{P}(\Gamma)$ of any rank greater than or equal to  $2N(\Gamma)-1$. The authors in \cite {bcggkmp}  
hint that a similar approach may provide a  sharper minimal dimension than 
$N(\Gamma)^2-2N(\Gamma)+2$ for the sphere $\mathbb{S}^5,$ and for $\mathbb{S}^{2n-1}$ in general. The case of  
$\mathbb{S}^5$ is answered in the next section. 

\section{Group-invariant CR maps of the sphere $\mathbb{S}^5$}

Our main theorem is the following.
\begin{theo}\label{theomain} 
Let $\Gamma$ be an admissible subgroup of $U(3)$. There exists an integer $n(\Gamma)< N(\Gamma)^2-2N(\Gamma)+2$ such that if $N\geq n(\Gamma)$ then there is a smooth nonconstant $\Gamma$-invariant CR map $f:\mathbb{S}^5 \to \mathbb{S}^{2N-1}$ for which $N$ is the minimal embedding dimension. More precisely:
\begin{enumerate}[i.]
\item If $\Gamma$ is equivalent to $\big\langle\omega I_2 \oplus\omega^2 I_1\big\rangle$ or $\big\langle\,\omega I_1 \oplus\omega^2 I_2\big\rangle$, where $\omega$ is a primitive $p^{th}$ root of unity with $p=2k+1$, then 
$$n(\Gamma)=2N(\Gamma)-1+k^2-1.$$
\item If $\Gamma$ is equivalent to  $\langle \omega I_3 \rangle$, where $\omega$ is a primitive $p^{th}$ root of unity, then 
$$n(\Gamma)=2N(\Gamma)-1+(p-2)p.$$
\item  If $\Gamma$ is equivalent to  $\langle \omega I_1 \oplus \omega^2 I_1 \oplus \omega^4 I_1 \rangle$, where $\omega$ is a $7^{th}$ root of unity, then 
$$n(\Gamma)=2N(\Gamma)-1+3\cdot 7+2=56.$$
\end{enumerate}
\end{theo}

\begin{remark}
The idea behind writing $n(\Gamma)$ as $2N(\Gamma)-1+\ldots$ in this theorem is for comparison purposes with the minimal rank $2N(\Gamma)-1$ obtained in Theorem 2 in \cite{bcggkmp}. 
\end{remark}

\begin{remark}
In the rest of the paper, we assume that  $\Gamma$ is actually equal to one of the four subgroups mentioned in Theorem \ref{theomain}. Indeed, we know from Grundmeier \cite{gr} that the minimal embedding dimension is invariant under coordinate change. 
\end{remark}

The strategy to prove Theorem \ref{theomain} follows directly from \cite{bcggkmp}. The key point is to construct polynomial maps in $\mathcal{P}(\Gamma)$ 
of ranks spanning a set
consisting of $N(\Gamma)-1$ consecutive integers, namely $\{n(\Gamma),n(\Gamma)+1,\ldots,n(\Gamma) +N(\Gamma)-2\}$. Reaching higher ranks is then obtained 
by multiplying the term of higher degree in $x_1$ of such a  polynomial by the canonical $\Gamma$-invariant polynomial $f_\Gamma$.
Although combinatorially more challenging, the iterations used to prove the first point i. of Theorem \ref{theomain} arise already in \cite{bcggkmp}. However, these 
iterations alone do not help cover all ranks after $n(\Gamma)$ in the case of $\langle \omega I_3 \rangle$. We also note that the proof of the iii. is purely ad-hoc and slightly different from the two others due to the variables having distinct weight. In any case, it is important for our approach to understand how any iteration in $\mathcal{P}(\Gamma)$ impacts the rank of 
a given polynomial.     

\subsection{Iterations and their action on the rank of polynomials}

Let $\Gamma$ be an admissible subgroup of $U(3)$. 
 We introduce important self maps of $\mathcal{P}(\Gamma)$
$$F_j,G_j: \mathcal{P}(\Gamma)\to \mathcal{P}(\Gamma),$$
$j=0,1,\ldots,p$. For a polynomial 
$P \in \mathcal{P}(\Gamma)$ containing the monomial $ x_1^{p-j}x_2^j$, $F_j(P)$ is obtained from $P$ by  replacing $x_1^{p-j}x_2^j$  by $x_1^{p-j}x_2^j \cdot f_{\Gamma}$ while  $G_j(P)$ is the polynomial obtained by  replacing the monomial $x_1^{p-j}x_2^j$ in $P$ by the 
term $\frac{1}{2}x_1^{p-j}x_2^j + \frac{1}{2}x_1^{p-j}x_2^j  \cdot f_{\Gamma}.$ In case $P$ does not contain the monomial $ x_1^{p-j}x_2^j$, we  set $F_j(P)=G_j(P)=P$. We also define  the map 
$$H:\mathcal{P}(\Gamma) \to \mathcal{P}(\Gamma)$$
where $H(P)$ is obtained from $P$ by replacing the monomial $x_3^p$ by $x_3^p\cdot f_\Gamma$. In case $P$ does not contain the monomial $x_3^p$, we define $H(P)=P$.

In what follows, we study the action of these maps on the rank of polynomials in $\mathcal{P}(\Gamma)$. More precisely, we define 
 for a nonnegative integer $m\leq p$ and $m_0=0,\ldots,m$, the following invariant polynomials
 \begin{equation}\label{eqf1}
 f_{-1,m}:= F_m\circ \ldots \circ F_1 \circ F_0(f_\Gamma),
 \end{equation}
 \begin{equation}\label{eqf2}
 f_{m_0,m} :=F_m\circ  \ldots F_{m_0+2} \circ F_{m_0+1} \circ G_{m_0} \circ \ldots \circ G_1 \circ G_0(f_\Gamma).
 \end{equation}

We start with the following lemma.
\begin{lemma}\label{lemran1}
Let $\Gamma= \langle \omega I_2 \oplus \omega^2 I_1 \rangle$ or $ \langle \omega I_1 \oplus \omega^2 I_2 \rangle$,  where $\omega$ is a $p^{th}$ root of unity with $p = 2k + 1$ an odd integer. For $m=0,\ldots, p$, we have 
$$\text{rank}(f_{-1,m})=2N(\Gamma)-1+m(k+1).$$ 
More precisely, for  $m=0,\ldots, p-1$
\begin{equation}\label{eqran}
\text{rank}(f_{-1,m+1})=\text{rank}(f_{-1,m})+k+1.
\end{equation}
\end{lemma}

As a direct corollary, we obtain: 
\begin{cor}\label{coran1}
Following the notations of Lemma \ref{lemran1}, we have  for $m=0,\ldots,p$ and $m_0=0,\ldots,m$  
$$\text{rank}(f_{m_0,m})=2N(\Gamma)+m_0(k+2)+(m-m_0)(k+1)=2N(\Gamma)+m_0+m(k+1).$$ 
\end{cor}

\begin{proof}[Proof of Lemma \ref{lemran1}] We first focus on  the subgroup $\Gamma= \langle \omega I_2 \oplus \omega^2 I_1 \rangle$. We set $p = 2k + 1$ for  $k \geq1$ and we consider  the invariant polynomial $f_{p,2}$ associated with the group $\langle \omega I_1 \oplus \omega^2 I_1 \rangle \leq U(2)$
given by 
$$f_{p,2}(x, y) = x^{p} + \sum_{j=1}^{\frac{p - 1}{2}} c_j x^{p - 2j} y^j + y^{p}.$$
The corresponding canonical $\Gamma$-invariant polynomial is  
\begin{eqnarray}\label{eqf}
f_{\Gamma}(x_1,x_2,x_3) & = & f_{p,2}(x_1 + x_2, x_3) \\\nonumber
& =&  x_1^{p} + \sum_{j=1}^{p} c_j x_1^{p - j} x_2^{j} + \sum_{j=1}^{\frac{p - 1}{2}} c_j (x_1 + x_2)^{p - 2j} x_3^{j} + x_3^{p}
\end{eqnarray}
and is of rank $N(\Gamma) = \frac{p^2 + 4p + 7}{4}.$ At this stage, we set all coefficients equal to one due to their irrelevance in the present proof. 
We now replace the monomial $x_1^{p} $ by $ x_1^{p} \cdot f_{\Gamma}$, leading to 
\begin{equation}\label{eqit1}
\begin{aligned}f_{-1,0} & = & 
x_1^{2p} + \sum_{j=1}^{p}  x_1^{2p - j} x_2^{j}  + \sum_{j=1}^{\frac{p - 1}{2}}  x_1^{p} (x_1 + x_2)^{p - 2j} x_3^{j}  + x_1^{p} x_3^{p} 
 \\
&& +  x_1^{p - 1} x_2+ \sum_{j=2}^{p} x_1^{p - j} x_2^{j}  + \sum_{j=1}^{\frac{p - 1}{2}}  (x_1 + x_2)^{p - 2j} x_3^{j}  + x_3^{p}.
\end{aligned}
\end{equation}
This new invariant  polynomial is of rank $2N(\Gamma)-1$ since there is no overlap between the newly introduced monomials and the previous ones.

\vspace{0.5cm}

To deal with $f_{-1,1}$, we first multiply the monomial  $x_1^{p - 1} x_2$ by $f_{\Gamma}$ and obtain 
$$
\underbrace{x_1^{2p - 1} x_2 
+ \sum_{j=1}^{p} x_1^{2p - j - 1} x_2^{j + 1}}_{I} 
+\underbrace{\sum_{j=1}^{\frac{p - 1}{2}} x_1^{p - 1} x_2 (x_1 + x_2)^{p - 2j} x_3^j}_{II} 
+ \underbrace{x_1^{p - 1} x_2 x_3^{p}}_{III}.
$$
We  then substitute in \eqref{eqit1} the term $x_1^{p - 1} x_2$ by the above. Most of the newly introduced monomials in I, II and II overlap with the  ones in $f_{-1,0}$ and we need to identify exactly which ones to accurately determine the rank of  $f_{-1,1}$.
First, note that the term III is a new monomial. The monomials in I appear already in the second term of \eqref{eqit1} with the exception of 
$ x_1^{p - 1} x_2^{p + 1}$. Thus the expression I introduces only one new monomial. We now focus on the term II and compare it with the only term of \eqref{eqit1} in which an overlap may occur, namely the third term of \eqref{eqit1} 
$$S_1 := \sum_{j=1}^{\frac{p - 1}{2}} x_1^{p} (x_1 + x_2)^{p - 2j} x_3^j.$$
Monomials in II and $S_1$ are respectively the form 
$x_1^{p +\ell- 1} x_2^{p - 2j-(\ell-1)} x_3^j$
and
$x_1^{p +\ell} x_2^{p - 2j-\ell} x_3^j $
with $j=1,\ldots,\frac{p-1}{2}, \ \ell=0,\ldots,p-2j.$ Thus, the term II introduces exactly $\frac{p-1}{2}=k$  new monomials, those the form $x_1^{p - 1} x_2^{p - 2j+1} x_3^j$,  $j=1,\ldots,\frac{p-1}{2}$.  
It follows that the rank of the new invariant polynomial  $f_{-1,1}$ is $2N(\Gamma)-1+(k+1)$. Moreover,  $f_{-1,1}$ is of the form 
\begin{eqnarray*} 
f_{-1,1}&=& x_1^{2p} + \sum_{j=1}^{p}  x_1^{2p - j} x_2^{j} +  x_1^{p - 1} x_2^{p + 1}+ \sum_{j=1}^{\frac{p - 1}{2}} x_1^{p} (x_1 + x_2)^{p - 2j} x_3^{j}  +
\sum_{j=1}^{\frac{p - 1}{2}}x_1^{p - 1} x_2^{p - 2j+1} x_3^j\\
&& + \sum_{\ell=0}^{1}x_1^{p - \ell} x_2^{\ell} x_3^{p}
+ \sum_{j=2}^{p} x_1^{p - j} x_2^{j}  
+ \sum_{j=1}^{\frac{p - 1}{2}}  (x_1 + x_2)^{p - 2j} x_3^{j}  
+ x_3^{p}.\\ 
\end{eqnarray*}

\medskip

Suppose now that after the $m^{th}$ iteration, with $m\leq p-1$, the invariant polynomial $f_{-1,m}$ has rank equal to $2N(\Gamma) - 1 + m(k+1)$ and is of the form
\begin{eqnarray*} 
f_{-1,m}&=& 
x_1^{2p} + \sum_{j=1}^{p}  x_1^{2p - j} x_2^{j} +  \sum_{\ell=1}^{m} x_1^{p - \ell} x_2^{p + \ell}+ \sum_{j=1}^{\frac{p - 1}{2}}  x_1^{p} (x_1 + x_2)^{p - 2j} x_3^{j} +\sum_{j=1}^{\frac{p - 1}{2}}\sum_{\ell=1}^{m}x_1^{p - \ell} x_2^{p - 2j+\ell} x_3^j\\
&& +\sum_{\ell=0}^{m}x_1^{p - \ell} x_2^{\ell} x_3^{p}+
 \sum_{j=m+1}^{p} x_1^{p - j} x_2^{j}  
+ \sum_{j=1}^{\frac{p - 1}{2}}  (x_1 + x_2)^{p - 2j} x_3^{j}  
+ x_3^{p}.\\ 
\end{eqnarray*}
We then replace in $f_{-1,m}$  the monomial
$x_1^{p - m - 1} x_2^{m+1}$ by the expression $x_1^{p - m - 1} x_2^{m+1} \cdot f_{\Gamma}$ which is equal to
$$
\underbrace{x_1^{2p - m - 1} x_2^{m+1} +\sum_{j=1}^{p} x_1^{2p - m-1- j} x_2^{m+1+j}}_{I} 
+\underbrace{\sum_{j=1}^{\frac{p - 1}{2}}   x_1^{p - m - 1} x_2^{m+1}(x_1 + x_2)^{p - 2j} x_3^{j} }_{II} 
+ \underbrace{ x_1^{p - m - 1} x_2^{m+1}x_3^{p}}_{III}.
$$
We treat these terms in a similar manner than we did  in the first iteration. The monomial III is new. All terms in the expression I appear in either the second or third terms of $f_{-1,m}$
with the exception of $x_1^{2p - m-1- j} x_2^{m+1+j}$ for $j=p$, that is, $x_1^{p - m-1} x_2^{p+m+1}$. Finally, most terms in II are present in either the fourth or fifth term of $f_{-1,m}$, and the $k$ new monomials are exactly $x_1^{p - (m+1)} x_2^{p - 2j+m+1} x_3^j$ for $j=1\ldots,\frac{p-1}{2}$. This proves \eqref{eqran}.

\vspace{0.5cm}

The proof in the case of the subgroup  $ \langle \omega I_1 \oplus \omega^2 I_2 \rangle$,  is similar and follows from the fact that the corresponding invariant polynomial is given by 
$f_{p,2}(x_1, x_2 + x_3),$ with $x_2$ and $x_3$ having the same weight. Note that $f_\Gamma$ is of rank $\frac{p^2 + 12p + 11}{8}.$

\end{proof}

\vspace{2mm}

\begin{lemma}\label{lemran2}
Consider the admissible subgroup $\Gamma = \langle \omega I_3 \rangle$ where $\omega$ is a primitive $p^{th}$ root of unity with $p \geq 2$. 
We have for $m=0,\ldots,p$ and $m_0=-1,0,\ldots,m$,  
$$\text{rank}(f_{m_0,m})=2N(\Gamma)+m_0+m\cdot p.$$ 
\end{lemma}

\begin{proof}
Similarly to Lemma \ref{lemran1} and Corollary \ref{coran1}, we will only prove that, for $m=0,\ldots,p$,  we have 
$$\text{rank}(f_{-1,m})=2N(\Gamma)-1+m\cdot p.$$ We consider  the canonical invariant polynomial  associated with  $\Gamma$ 
$$ f_{\Gamma}(x_1, x_2, x_3)=f_{\Gamma,1}(x_1 +x_2, x_3)=\sum_{\begin{subarray}{c}j+k+\ell=p\end{subarray}}\frac{p!}{\,j!\,k!\ell!}\,x_{1}^{j}x_{2}^{k}x_{3}^{\ell}$$
where  $f_{\Gamma,1}(x,y)=(x+y)^p$. In the above sum and all those below, the indices $j,k$ and $\ell$ are nonnegative. 
The polynomial  $f_{\Gamma}$ is of rank  $N(\Gamma) = \binom{p + 2}{2} = \frac{(p + 1)(p + 2)}{2}$. We also set all coefficients equal to one. We now replace  
$x_{1}^{p}$ by $x_{1}^{p}\cdot f_{\Gamma}$ and get 
\begin{equation*}
\begin{aligned}f_{-1,0} & =  & \sum_{\begin{subarray}{c}j+k+\ell=p\end{subarray}}x_{1}^{p+j}x_{2}^{k}x_{3}^{\ell} + \sum_{\begin{subarray}{c}j+k+\ell=p\\ j\neq p\end{subarray}}x_{1}^{j}x_{2}^{k}x_{3}^{\ell} 
\end{aligned}
\end{equation*}
which is  of rank $2N(\Gamma)-1=(p+2)(p+1)-1$. 

By a direct computation, we obtain
\begin{equation*}
\begin{aligned}f_{-1,1} & =  & \sum_{\begin{subarray}{c}j+k+\ell=p\end{subarray}}x_{1}^{p+j}x_{2}^{k}x_{3}^{\ell} +\sum_{\begin{subarray}{c}j+k+\ell=p\end{subarray}}x_{1}^{p+j-1}x_{2}^{k+1}x_{3}^{\ell} + \sum_{\begin{subarray}{c}j+k+\ell=p\\ (j,\ell)\neq (p,0),(p-1,1)\end{subarray}}x_{1}^{j}x_{2}^{k}x_{3}^{\ell} 
\end{aligned}
\end{equation*}
The polynomial $f_{-1,1}$  contains  $p+1$ linearly independent  monomials that do not appear in  $f_{-1,0}$, namely the terms in the second sum for which $j=0$; all other terms in the second sum appear in the first. Thus, and taking into account the term $x_1^{p-1}x_2$ in $f_{-1,0}$ that has been substituted, the rank of $f_{-1,1}$ is  $2N(\Gamma)-1+p$.

We claim that the polynomial $f_{-1,m}$, $m=2,\ldots,p$ is  of rank $2N(\Gamma)-1+m\cdot p$. Indeed, the new contributions in the polynomial
\begin{equation}
\label{eqitm}
\begin{aligned}
f_{-1,m}  = & \sum_{\begin{subarray}{c}j+k+\ell=p\end{subarray}}x_{1}^{p+j}x_{2}^{k}x_{3}^{\ell} + \sum_{s=1}^{m-1} \sum_{\begin{subarray}{c}j+k+\ell=p\end{subarray}}x_{1}^{p+j-s}x_{2}^{k+s}x_{3}^{\ell} +  \sum_{\begin{subarray}{c}j+k+\ell=p\end{subarray}}x_{1}^{p+j-m}x_{2}^{k+m}x_{3}^{\ell} \\
\\
&  +\sum_{\begin{subarray}{c}j+k+\ell=p\\ \ (j,\ell)\neq (p,0),(p-1,1),\ldots, (p-m,m),\end{subarray}}x_{1}^{j}x_{2}^{k}x_{3}^{\ell} 
\end{aligned}
\end{equation}
 arise from the $p+1$  terms in the  third sum  for which $\ell=0$.

\end{proof}

Finally, we  need the following lemma in relation with the iteration $H$ defined earlier. 
\begin{lemma}\label{lemran3}
Following the notations of Lemma \ref{lemran2}, we have, for $m=2,\ldots, p$ and $m_0=-1,0,\ldots,m$,  
$$\text{rank}(H(f_{m_0,m}))=\text{rank}(f_{m_0,m})+N(\Gamma)-2-m.$$ 
\end{lemma}
\begin{proof}

Fix $m \in \{2,\ldots,p\}$ and first consider the polynomial $f_{-1,m}$ given by \eqref{eqitm}. It follows that $H(f_{-1,m})$ is of the form 
\begin{eqnarray*}
H(f_{-1,m}) & =  & \sum_{\begin{subarray}{c}j+k+\ell=p\end{subarray}}x_{1}^{p+j}x_{2}^{k}x_{3}^{\ell} + \sum_{s=1}^{m} \sum_{\begin{subarray}{c}j+k+\ell=p\end{subarray}}x_{1}^{p+j-s}x_{2}^{k+s}x_{3}^{\ell} \\
\\
& & +\sum_{\begin{subarray}{c}j+k+\ell=p\\ (j,\ell)\neq (p,0),(p-1,1),\ldots, (p-m,m)\end{subarray}}x_{1}^{j}x_{2}^{k}x_{3}^{\ell} 
+  \sum_{\begin{subarray}{c}j+k+\ell=p\end{subarray}}x_{1}^{j}x_{2}^{k}x_{3}^{\ell+p}.
\end{eqnarray*}
In the last sum, which contains $N(\Gamma)$ terms, exactly $m+1$ of them occur already in the first four sums of  \eqref{eqitm}. Indeed, the monomial $x_{1}^p x_{3}^p$ is contained in the first sum, while $x_1^{p-s}x_2^sx_3^p, s=1,\ldots, m-1$, appear in the second sum and  
$x_1^{p-m}x_2^mx_3^p$ occurs in the third sum.  

The proof is similar for polynomials $f_{m_0,m}, m_0=0,\ldots, m$ since, keeping in mind that we set all coefficients equal to one, we have 
$$f_{m_0,m}=f_{-1,m}+\sum_{j=0}^{m_0}x_1^{p-j}x_2^j.$$ 
\end{proof}

\subsection{Proof of Theorem \ref{theomain}}

We recall that in order to prove Theorem \ref{theomain}, we need to construct polynomials in $\mathcal{P}(\Gamma)$ 
of ranks spanning a set consisting of $N(\Gamma)-1$ consecutive integers, namely $\{n(\Gamma),n(\Gamma)+1,\ldots,n(\Gamma) +N(\Gamma)-2\}$. 

\begin{proof}
We start with i. 
Due to their similarities, we focus on the group $\Gamma=\big\langle\omega I_2 \oplus\omega^2 I_1\big\rangle$.  Starting from the canonical polynomial $f_\Gamma$ \eqref{eqf}, we  consider the  polynomials $f_{m_0,m} \in \mathcal{P}(\Gamma)$, $m=0,\ldots,p$ and $m_0=-1,0,\ldots,m$ defined in \eqref{eqf1} and \eqref{eqf2}. According to Lemma \ref{lemran1} and Corollary \ref{coran1}, 
the $k+1$ polynomials $f_{m_0,k-1}, m_0=-1,0,\ldots,k-1$ have ranks spanning 
$$\{2N(\Gamma)+k^2-2, 2N(\Gamma)+k^2-1,2N(\Gamma)+k^2,\ldots, 2N(\Gamma)+k^2+k-2\}.$$
It follows that for $m=k-1,\ldots,p-2$ and $m_0=-1,0\ldots,m$, the ranks of the polynomials 
$f_{m_0,m}$ and $f_{-1,p-1},\ldots,f_{k-1,p-1}$   cover a set of $N(\Gamma)-1=(k+2)(k+1)$ consecutive integers. Indeed, the polynomial 
$f_{-1,k-1}$ has the lowest rank $2N(\Gamma)+k^2-2$ and $f_{k-1,p-1}$ has the highest rank 
$$2N(\Gamma)+k-1+2k(k+1)=2N(\Gamma)+k^2-2+N(\Gamma)-2,$$
and all ranks in between are achieved, which ends the proof of i. 

\vspace{0.5cm}

We now prove the second statement ii. Recall that according to Lemma \ref{lemran2}, the rank of the polynomial $f_{m_0,m}$ for $m_0=-1,0,\ldots,m$ and $m=0,\ldots,p$ is equal to   $2N(\Gamma)+m_0+m\cdot p.$ 
We note that for $m=0,\ldots,p$, the $m+2$ polynomials $f_{m_0,m}, m_0=-1,\ldots,m$ have ranks spanning the set 
$$\{2N(\Gamma)-1+m\cdot p, 2N(\Gamma)+m\cdot p,\ldots, 2N(\Gamma)+m+m\cdot p\}.$$
Thus, for $p-2 \leq m\leq p$, the ranks of these polynomials  $f_{m_0,m}, m_0=-1,\ldots,m$ cover the set of $3p+2$ consecutive integers 
$$\{2N(\Gamma)-1+(p-2)p, 2N(\Gamma)+(p-2) p,\ldots, 2N(\Gamma)+p+p^2\},$$
Notice that $2N(\Gamma)+p+p^2+1=2N(\Gamma)-1+(p-2)p+3p+2$. We now claim that the  integers 
$$2N(\Gamma)-1+(p-2)p+3p+2,\ldots, 2N(\Gamma)-1+(p-2)p+N(\Gamma)-2$$ 
are also ranks of polynomials arising from $f_{m_0,m}$ with $\lfloor p/2\rfloor-1 \leq m\leq p-2$. Indeed, from Lemma \ref{lemran3} the rank of the polynomial $H(f_{m_0,m})$ is equal to  
$$2N(\Gamma)-1+m\cdot p+N(\Gamma)+m_0-m-1.$$ In addition, the rank of $F_0(f_{m_0,m})$  is equal to 
$$2N(\Gamma)-1+m\cdot p+N(\Gamma)+m_0.$$
Consider an integer $2N(\Gamma)-1+(p-2)p+\ell=2N(\Gamma)-1+p^2-2p+\ell$ with $3p+2 \leq \ell \leq N(\Gamma)-2$. According to the Euclidean division, we can write: 
$$p^2-2p+\ell-N(\Gamma)=q\cdot p + r$$
with $\lfloor p/2\rfloor-1 \leq q \leq p-3$  and $0\leq r < p$. In case $0\leq r \leq q$, the integer $2N(\Gamma)-1+(p-2)p+\ell$ is then the rank of the polynomial  $F_0(f_{r,q})$. If instead we have $q<  r \leq p-1$,then the integer $2N(\Gamma)-1+(p-2)p+\ell$ is the rank of 
the polynomial $H(f_{m_0,q+1})$ with $m_0=q+2+r-p$ that satisfies $0 \leq m_0\leq q+1.$ 
This achieves the proof of ii.

\vspace{0.5cm}
Finally, we move to iii.
The canonical polynomial, of rank $N(\Gamma)=17$,  is given by 
\begin{eqnarray*}
f_{\Gamma}(x_1, x_2, x_3) & = & x_1^7 + 7x_1^5x_2 + 14x_1^3x_2^2 + 7x_1x_2^3  + 7x_1^3x_3 + 14x_1x_2x_3    \\
\\
&&  + x_2^7 + 7x_1^2x_2^4x_3 + 7x_2^5x_3 +7x_1^4x_2x_3^2  + 7x_1^2 x_2^2 x_3^2+14x_2^3x_3^2\\
\\
& & + 14x_1^2  x_3^3+ 7x_2 x_3^3 \;+ 7x_1 x_2^2 x_3^4+ 7x_1  x_3^5+ x_3^7.\\
\end{eqnarray*}

Note that $f_{\Gamma}$ is the sum of weighted homogeneous monomials  in that case. The method used for the three previous groups does not apply directly due to the fact that each variable in $f_{\Gamma}$ has a different weight, namely $1$ 
for $x_1$, $2$ for $x_2$ and $4$ for $x_4$. Instead, we follow an ad-hoc approach. The previous iterative self maps $F,G$ and $H$ of $\mathcal{P}(\Gamma)$ - and their action on the ranks - are not entirely convenient in this case and, instead of introducing new weighted maps, we simply deal with seven ad-hoc iterations to produce polynomials with ranks spanning $N(\Gamma)-1=16$ consecutive integers. In what follows, rank computations are straightforward, focusing only on powers thought to be triple $(j,k,\ell)$ occurring in the corresponding polynomials.  

The first four iterations are relatively standard and are focused on the monomials $x_1^{7-2j}x_2^{j}$ for $j=0,\ldots,3$. We first consider the new polynomials $f_{-1,0}=F_0(f_\Gamma)$ and $f_{0,0}=G_0(f_\Gamma)$ of respective ranks $33$ and $34$. We now denote  by $f_{-1,1}$ the polynomial obtained from $f_{-1,0}$ by replacing $x_1^5x_2$ by  $x_1^5x_2 \cdot f_\Gamma$; doing so in $f_{0,0}$ leads to  $f_{0,1}$. Moreover, replacing $x_1^5x_2$ by $\frac{1}{2} x_1^5x_2 + \frac{1}{2} x_1^5x_2\cdot f_\Gamma$  in $f_{0,0}$ provides a polynomial denoted by $f_{1,1}$. We introduce  new polynomials, namely $f_{m_0,m}, m=2,3, \ m_0=-1,\ldots,m$, repeating this idea on $x_1^3x_2^2$ and $x_1x_2^3$.  After straightforward computations, we note the following: the ranks of  $f_{m_0,1}, m_0=-1,0,1$ are $41, 42$ and $43$. Those of $f_{m_0,2}, m_0=-1,\ldots,2$ span the set $\{49,\ldots,52\}$, while those of $f_{m_0,3}, m_0=-1,\ldots,3$ cover $\{56,\ldots,60\}$. The next six polynomials $f_{m_0,4}, \ m_0=-1,\ldots,4$  are obtained by applying the same procedure on $x_1x_2x_3$ and  $f_{m_0,3}$, that is, replacing $x_1x_2x_3$ by $x_1x_2x_3\cdot f_\Gamma$  or $\frac{1}{2}x_1x_2x_3+\frac{1}{2}x_1x_2x_3\cdot f_\Gamma$. Doing so gives the ranks $\{58,\ldots,63\}$. Considering now $x_1^2x_3^3$ gives seven new polynomials $f_{m_0,5}, \ m_0=-1,\ldots,5$ of ranks spanning $\{62,\cdots 68\}$.   

At this stage, we have obtained ranks that cover  $\{56,\ldots,68\}$. We then consider $x_2x_3^3$ to obtain the next eight polynomials $f_{m_0,6}, \ m_0=-1,\ldots,6,$ and note that $f_{m_0,6}$ for $m_0=2,3,4$ have respective ranks $69,70$ and $71$.

\end{proof}

\noindent {\bf Acknowledgments} This work was  supported by the Center for Advanced Mathematical Sciences and by a URB grant from the American University of Beirut.

\vskip 0.5cm
{\small
\noindent Mona Al Batrouni, Florian Bertrand\\
Department of Mathematics,\\
American University of Beirut, Beirut, Lebanon\\
{\sl E-mail address}: mab108@mail.aub.edu, fb31@aub.edu.lb\\

\end{document}